\documentclass [11pt]{amsart}
\usepackage{amsmath}
\usepackage{amssymb}
\usepackage{amscd}
\usepackage{epsfig}
\usepackage{amsxtra}
\usepackage{pifont}
\usepackage{mathabx}
\usepackage{MnSymbol}
\usepackage{accents}
\usepackage[none]{hyphenat}
\usepackage{scalerel,stackengine}
\usepackage{tikz}
\usetikzlibrary{cd}
\usepackage{caption}
\usepackage{graphicx}
\usetikzlibrary{arrows.meta}
\usetikzlibrary{arrows,chains,matrix,positioning,scopes}
\usetikzlibrary{decorations.pathreplacing}
\usepackage{float}
\usepackage{csquotes}
\stackMath
\newcommand\reallywidehat[1]{%
\savestack{\tmpbox}{\stretchto{%
  \scaleto{%
    \scalerel*[\widthof{\ensuremath{#1}}]{\kern-.6pt\bigwedge\kern-.6pt}%
    {\rule[-\textheight/2]{1ex}{\textheight}}
  }{\textheight}%
}{0.5ex}}%
\stackon[1pt]{#1}{\tmpbox}%
}

\newcommand\reallywidecheck[1]{%
\savestack{\tmpbox}{\stretchto{%
  \scaleto{%
    \scalerel*[\widthof{\ensuremath{#1}}]{\kern-.6pt\bigwedge\kern-.6pt}%
    {\rule[-\textheight/2]{1ex}{\textheight}}
  }{\textheight}%
}{0.5ex}}%
\stackon[1pt]{#1}{\scalebox{-1}{\tmpbox}}%
}

\numberwithin{equation}{section}

\allowdisplaybreaks

\newcommand{\NN}{\mathbb N}

\title[Substitutions, Bratteli Diagrams and Algebras]{Symbolic Substitutions, Bratteli Diagrams and Operator Algebras}

\author{Dylan Gawlak}
\address{Department of Mathematics and Statistics, MacEwan University, Edmonton, AB, Canada}
\email{gawlakd@mymacewan.ca}

\author{Christopher Ramsey}
\email{ramseyc5@macewan.ca}

 \newtheorem{theorem}{Theorem}[section]
 \newtheorem{lemma}[theorem]{Lemma}
 \newtheorem{proposition}[theorem]{Proposition}
 \newtheorem{corollary}[theorem]{Corollary}

 \newtheorem{definition}[theorem]{Definition}
 \newtheorem{example}[theorem]{Example}
  \newtheorem{remark}[theorem]{Remark}


\begin{document}
\maketitle
\begin{abstract}
In this paper we look at symbolic substitutions and their relationship to Bratteli diagrams and their associated operator algebras. In particular, we consider the equivalence relation on substitutions induced by telescope equivalence of Bratteli diagrams. Such an equivalence preserves pure aperiodicity and primitivity but fails to preserve rank, order, and number of letters. In a similar manner, we consider the equivalence relation on substitutions induced by telescope equivalence of ordered Bratteli diagrams. This results in a finer equivalence but fails to provide a complete invariant. An application to Fibonacci-like substitutions is developed.
\end{abstract}

\section{\textbf{Introduction}}

A particular family of infinite directed graphs was introduced by Ola Bratteli \cite{ILFD} to study approximately finite-dimensional (AF) algebras, that is, direct limits of directed sequences of finite-dimensional C*-algebras. 
Since there introduction, Bratteli diagrams have found applications in areas such as topological dynamical systems and symbolic substitutions \cite{AA}, the latter of which is the focus in this paper.

A symbolic substitution is a function from a finite set of letters $\mathcal L$ into the set of all finite words made by the letters in $\mathcal L$. This function can be extended to words made from the letters of $\mathcal L$ by applying the substitution to the individual letters in the word and then concatenating them together. Symbolic substitutions have seen application in areas such as the study of quasicrystals and model sets \cite{AO},  dynamical systems \cite{SDA}, dynamics and coding \cite{LindMarcus}, theoretical physics \cite{SDO}, and logic \cite{LS}, to name a few.

Bratteli diagrams are a natural tool for studying symbolic substitutions as we are able to easily encode information about the substitution matrix corresponding to the substitution rule into a Bratteli diagram. 
In this paper we shall focus on the equivalence relation induced on symbolic substitutions by telescope equivalence of their associated stationary Bratteli diagrams, unordered and ordered. Some properties of substitutions survive, for example simplicity, primitivity, and aperiodicity, but both equivalences collapse many distinct substitutions into one class. Some of these problems have been addressed in  \cite{NSB} and \cite{DIP}. The conclusion of the paper gives a case study of symbolic substitutions supported on powers of the Fibonacci matrix. Telescope equivalence of ordered Bratteli diagrams is shown to be a very natural equivalence relation on this set which is distinct from local indistinguishability. 

We aim to make this a detailed resource and include lots of explanations and examples. To this end, sections 2 and 3 contain the necessary background material on symbolic substitutions and Bratteli diagrams, respectively. In section 4 we study symbolic substitutions through telescope equivalence of their associated Bratteli diagrams. This section culminates in Proposition 4.6 which proves that every primitive substitution matrix has a Bratteli diagram that is telescope equivalent to an infinite collection of Bratteli diagrams associated to distinct primitive substitution matrices. Section 5 studies the invertibility of symbolic substitution matrices. Example 5.2 and Proposition 5.4 together establish that there is a non-invertible substitution matrix whose Bratteli diagram is not telescope equivalent to any diagram arising from an invertible substitution matrix. In section 6 we move into a finer equivalence relation for symbolic substitution matrices, namely telescope equivalence of ordered Bratteli diagrams. In particular, it is established that this equivalence preserves maximal and minimal paths (Corollary 6.5 and Proposition 6.6), which is used to distinguish several variations of the Fibonacci substitution. Proposition 6.16 establishes that any symbolic substitution that is equivalent in this way to that of a power of the Fibonacci matrix must itself be a power of the Fibonacci matrix. The remainder of the section further explores these Fibonacci symbolic substitutions.



\section{\textbf{Symbolic Substitutions}}
We will now turn to symbolic substitutions. For a deeper exploration of this subject look at \cite{AO}.
\begin{definition}
Let $\mathcal L$ be a finite, nonempty set and call the elements of $ \mathcal L$ letters. Let $ \mathcal L^{*}$ be the set of all finite words of elements from $ \mathcal L$. A \textbf{\textit{symbolic substitution}} is any mapping $\sigma :\mathcal L\rightarrow \mathcal L^{*}$. A symbolic substitution can be applied to elements of $\mathcal L^{*}$ by concatenating the substitutions of the individual letters. 
\end{definition}
The most famous example is the Fibonacci substitution: 
\begin{align*}
    \sigma(l_1)&=l_1l_2 \\
    \sigma(l_2)&=l_1
\end{align*}
The name arises because the Fibonacci sequence $(f_n)_{n\geq 0}$ is naturally encoded into this substitution rule. In particular, by applying $\sigma$ $n$ times to $l_1$, the resulting word will have length $f_{n+2}$. Similarly, applying $\sigma$ $n$ times to $l_2$ yields a word of length $f_{n+1}$. The substitution rule also shares the same recursive nature in that $\sigma^n(l_1)=\sigma^{n-1}(l_1)\sigma^{n-2}(l_1)$, $n\geq 2$.

In the study of symbolic substitutions, an important tool is the substitution matrix $M$ associated to a given symbolic substitution.
\begin{definition}
Let $\sigma$ be a substitution on n letters $\{l_1,l_2,...,l_n\}$. The \textbf{\textit{substitution matrix}} $M$ associated with $\sigma$ is the matrix whose $(i,j)$-entry is the number of copies of $l_j$ inside $\sigma(l_i)$.
\end{definition}

From the definition, the substitution matrix is unique. However, there usually are multiple different substitutions leading to the same matrix. This is because the order of the letters in the substitution rule is ignored. For example: \begin{align*}
    \sigma(l_1)&=l_2l_1 \\
    \sigma(l_2)&=l_1
\end{align*}
has the same substitution matrix, $\begin{bmatrix} 1 & 1 \\ 1 & 0
\end{bmatrix}$, as the Fibonacci substitution.  

One thing to note is that this definition implies that the matrices compose ``backwards". That is, if $\sigma_M$ and $\sigma_N$ are symbolic substitutions with substitution matrices $M$ and $N$ then $\sigma_M\circ\sigma_N$ has the substitution matrix $NM$.

\begin{definition}
A non-negative $n\times n$ matrix M is called \textbf{\textit{primitive}} if there exists $k\in \mathbb{N}$ such that all of the entries of $M^k$ are strictly greater than zero.
\end{definition}

\begin{definition}
A substitution rule $\sigma$ on a finite alphabet $\mathcal L$ is called \textbf{\textit{primitive}} if there exists some $k\in \mathbb{N}$ such that for every pair $(i,j)$ $1\leq i,j\leq n$, $l_j$ is a subword of $\sigma^k(l_i)$
\end{definition}

It is trivial to see that a substitution rule is primitive if and only if the corresponding substitution matrix is also primitive. In this paper, primitive substitutions will be the main focus as this allows us to use the Perron-Frobenius theorem.

\begin{theorem}
(Perron-Frobenius Theorem): If $M$ is a primitive matrix, then $M$ has a real eigenvalue that has multiplicity one and has modulus greater than any other eigenvalue of $M$. This eigenvalue is called the \\ \textbf{\textit{PF-eigenvalue}}, and is denoted $\lambda_{PF}$.
\end{theorem}

\begin{remark}
One can associate to such a primitive symbolic substitution a discrete set of points on the real line that encode the substitution. Specifically, the left PF-eigenvector, $v_{PF}$, defines a geometric inflation rule by associating to the letter $l_i$ an interval of length $(v_{PF})_i$. The geometric inflation rule is then realized by scaling the interval by a factor of $\lambda_{PF}$ and then dissecting the interval into copies of the original intervals based on the substitution rule \cite{AO}. The right PF-eigenvector on the other hand can be normalized to give us the relative asymptotic frequencies of the letter \cite{AO}.  
\end{remark}

The following few definitions and their theory can be found in can be found in a number of place, e.g. \cite{CSS}.
\begin{definition}
A word $w\in \mathcal L^*$ is said to be a \textbf{\textit{factor}} (or \textbf{\textit{subword}}) of $w'\in \mathcal L^*$ if there exist, possibly empty, words $u,v\in \mathcal L^*$ such that $uwv=w'$.
\end{definition}
\begin{definition}
Let $\sigma$ be a substitution rule on a finite alphabet $\mathcal L$. A word $w\in \mathcal L^*$ is called \textbf{\textit{admissible}} if there exists $l\in \mathcal L$ and $p\in \mathbb{N}$ such that $w$ is a factor of $\sigma^p(l)$.
\end{definition}
\begin{definition}
A bi-infinite sequence $s\in \mathcal L^{\mathbb{Z}}$ is said to be \textbf{\textit{admitted}} by $\sigma$ if every factor of $s$ is admissible  for $\sigma$.
\end{definition}

\begin{definition}
A substitution rule $\sigma$ is said to be \textbf{\textit{aperiodic}} if every admitted bi-infinite sequence in $\mathcal L^{\mathbb Z}$ is not periodic.
\end{definition}
 
Informally, an aperiodic substitution rule can be thought of as a substitution rule which lacks translational symmetry. The following is a well known theorem for primitive substitutions.

\begin{theorem}[Theorem 4.6, \cite{AO}]\label{thm:aperiodic}
Let $\sigma$ be a primitive substitution rule on a finite alphabet with substitution matrix $M$. If the PF-eigenvalue of M is irrational then the symbolic substitution is aperiodic. 
\end{theorem}

If a primitive substitution matrix $M$ has an irrational PF-eigenvalue we shall say that it is a \textit{\textbf{purely aperiodic}}  substitution matrix. Do note that the converse to the previous theorem is false.

Substitution tilings can be generalized in multiple ways with the most common being geometric tiling substitutions in the Euclidean plane. This is essentially a higher-dimensional version  of the geometric inflation rule for symbolic substitutions, where the letters are replaced with prototiles, which are subsets of $\mathbb{R}^2$ that are homeomorphic to the unit ball. More details about this can be found in \cite{EP} and \cite{AO}. In this paper, we will be focusing on symbolic substitutions exclusively. 

\section{\textbf{Bratteli Diagrams}}
In this section, we introduce Bratteli diagrams and their associated AF-algebras. These ideas were first introduced by Bratteli in \cite{ILFD}. 

\begin{definition}
A C*-algebra $\mathcal A$ is said to be an \textbf{\textit{AF-algebra}} (Approximately Finite-dimensional algebra) if there exists a sequence of finite-dimensional C*-algebras $\{\mathcal A_i\}$ and $*$-homomorphisms $\phi_i : \mathcal A_i \rightarrow \mathcal A_{i+1}$ such that $\mathcal A$ is the closure of the algebraic direct limit of the chain system $(\mathcal A_i,\phi_i)$, $i\in \mathbb{N}$.
\end{definition}

Recall that a $*$-homomorphism $\phi:\mathcal A\rightarrow \mathcal B$ between C$^*$-algebras is an algebra homomorphism that preserves adjoints, namely $\phi(a^*) = \phi(a)^*$ for all $a\in \mathcal A$.

All finite-dimensional C*-algebras are $\ast$-isomorphic to the direct sum of full matrix algebras. As a result, if $\mathcal A_i \cong \bigoplus_{i=1}^{n_1} M_{m_i}\subseteq M_s$ were $s=\sum_{i=1}^{n_1}m_i$ and $\mathcal A_{i+1} \cong \bigoplus_{i=1}^{n_2}M_{m'_i}\subseteq M_t$ were $t=\sum_{i=1}^{n_2}m'_i$  it can be shown that the $\ast$-homomorphism $\phi_i:\mathcal A_i \rightarrow \mathcal A_{i+1}$ is inner equivalent to the embedding of copies of the full matrix algebras of $\mathcal A_i$ into the full matrix algebras of $\mathcal A_{i+1}$ as block diagonal entries along with zero matrices when needed. $\ast$-homomorphisms of this form are said to be canonical.  

\begin{example}
Consider $\phi (\mathbb{C}\oplus M_2)\rightarrow M_3\oplus M_5 $ given by: $$\phi (\mathbb{C}\oplus M_2)=\begin{bmatrix}\mathbb{C} & 0 \\ 0 & M_2
\end{bmatrix} \oplus \begin{bmatrix}
M_2 & 0 \\ 0 & 0
\end{bmatrix} $$ where the zero matrices are of the appropriate sizes. 
\end{example}
In the above example, the embedding is not unital, but from here on we shall always assume that the embeddings are unital. 
Because $\ast$-homomorphisms are inner equivalent to these canonical $\ast$-homomorphisms, we are able to associate a matrix to a $\ast$-homomorphism called the transition matrix.

\begin{definition}
Let $\mathcal A_1 \cong \bigoplus_{i=1}^{n_1}M_{m_i} $ and $\mathcal A_2 \cong \bigoplus_{i=1}^{n_2}M_{m'_i}$ be two finite-dimensional C*-algebras and let $\phi$ be a canonical $\ast$-homomorphism from  $\mathcal A_1$ into $\mathcal A_2$. We define the \textbf{\textit{transition matrix}} $B$ to be an $n_1\times n_2$ matrix where the $(i,j)$-entry denotes the number of copies of $M_{m_i}$ inside of $M_{m'_j}$.

\end{definition}

Using the transition matrices, we are able to associate AF-algebras with special types of infinite directed graphs called labelled Bratteli diagrams.

\begin{definition}
A \textbf{\textit{(labelled) Bratteli diagram}} $(V,E, r, s,d )$ is an infinite directed graph $(V,E, r, s)$ with a labelling $d$ which satisfies the following properties: \\
1. The vertex set $V$ and edge set $E$ can be written as the countable union of nonempty pairwise disjoint finite sets $V_n$, $n\geq 0$ and $E_n$, $n\geq 1$  \\
2. There exist range and source maps $r$ and $s$ respectively such that for any edge $e\in E_n$ then $s(e)\in V_{n-1}$ and $r(e)\in V_n$. Additionally $r^{-1}(v)\neq \varnothing$ and $s^{-1}(v)\neq \varnothing$ with the exception of $v\in V_0$ \\ 
3. There exists a map $d:V\rightarrow \mathbb{N}$, called the labelling such that for any vertex $v$ not in $V_0$, $d(v)= \displaystyle{\sum_{r(e)=v}}d(s(e))$ 
\end{definition}
Note that this definition of a Bratteli diagram is slightly more restrictive than in some other papers as it corresponds to chain systems with unital embeddings. 

The association between diagram and algebra is fairly straightforward. 
In particular, for an algebraic chain system of finite-dimensional C*-algebras $(\mathcal A_n,\phi_n)$, where $\mathcal A_n\cong \bigoplus_{i=1}^{m_n}M_{s_i}$, each full matrix algebra gets a vertex associated to it, with the labeling denoting the size of the matrix algebra. The edges between levels are assigned based on the transition matrices. To any given AF-algebra, we can associate infinitely many Bratteli diagrams. However, any individual Bratteli diagram corresponds to only one AF-algebra, see Proposition \ref{BratteliAF}. 

\begin{definition}
Two Bratteli diagrams are said to be isomorphic if there exist bijections between the vertices and edges such that the labelling, range and source maps are preserved.
\end{definition}

Isomorphism then corresponds to permuting the vertices or equivalently applying a permutation conjugation to the transition matrix at every level.
For example, the following two diagrams are isomorphic:
\begin{figure}[H]
\begin{tikzcd}
1 \arrow[r] \arrow[rd] & 2 \arrow[r] \arrow[rd] & 3 \arrow[r] \arrow[rd]  & \cdots \\
1 \arrow[ru]           & 1 \arrow[ru]           & 2 \arrow[ru]                     & \cdots
\end{tikzcd}
\quad \quad \quad
\begin{tikzcd}
1 \arrow[rd]           & 1 \arrow[rd]           & 2 \arrow[rd]                    & \cdots \\
1 \arrow[ru] \arrow[r] & 2 \arrow[ru] \arrow[r] & 3 \arrow[ru] \arrow[r]  & \cdots
\end{tikzcd}
\end{figure}


\begin{definition}
Let $(V,E,r,s,d)$ be a Bratteli diagram and  let $(m_n)_{n\in \mathbb{N}}$ be a strictly increasing sequence in $\mathbb{N}$. The \textbf{\textit{telescoping}} of $(V,E,r,s,d)$ with respect to $m_{n}$ is the Bratteli diagram $(V',E',r',s',d')$ where $V_{n}'=V_{m_{n}}$ and $E'_n$ is determined by the set of all paths from $V_{m_{n-1}}$ to $V_{m_{n}}$. $d'$ is the restriction of $d$ to $V'$. $r'$ and $s'$ are the extensions of $r$ and $s$ to the paths that define $E'_{n}$. 
\textbf{\textit{Microscoping}} is the inverse of telescoping.
\end{definition}

For example, the telescoping of the above diagram to its even levels gives:

\begin{figure}[h]
\begin{tikzcd}
1 \arrow[r, shift right] \arrow[r, shift left]  & 3 \arrow[r, shift right] \arrow[r, shift left] \arrow[rd] & 8 \arrow[r, shift right] \arrow[r, shift left] \arrow[rd] & 21 \arrow[r, shift right] \arrow[r, shift left] \arrow[rd] & 55 \arrow[r, shift right] \arrow[r, shift left] \arrow[rd] & \cdots \\
1 \arrow[ru] \arrow[r]             & 2 \arrow[ru] \arrow[r]             & 5 \arrow[ru] \arrow[r]             & 13 \arrow[ru] \arrow[r]             & 34 \arrow[ru] \arrow[r]             & \cdots
\end{tikzcd}
\end{figure}

It is not hard to see that the transition matrix for this Bratteli diagram is $\begin{bmatrix} 2 & 1 \\ 1 & 1
\end{bmatrix}=\begin{bmatrix}1 & 1 \\ 1 & 0  \end{bmatrix}^2$
In general, telescoping corresponds to matrix multiplication, while microscoping corresponds to matrix factorization. 

\begin{definition}
Two Bratteli diagrams are said to be \textbf{\textit{telescope equivalent}} if you can obtain one from the other by a finite number of telescopings, microscopings, and isomorphisms. 
\end{definition}

\begin{theorem}[Bratteli \cite{ILFD}]\label{BratteliAF}
Let $\mathcal A$ and $\mathcal B$ be AF-algebras. Then $\mathcal A$ is $*$-isomorphic to $\mathcal B$ if and only if their Bratteli diagrams are telescope equivalent.
\end{theorem}

Of note is that Elliott proved that the above statement is true if and only if their associated \textit{dimension groups} are isomorphic \cite{Ell}. While much of the theory in this paper can be viewed and proven in the context of dimension groups the authors have chosen not to proceed in this manner. The interested reader can look to \cite{BMT} and \cite{LindMarcus}.


This concept of telescope equivalence is very useful as it allows one to turn the question of $\ast$-isomorphism between AF-algebras into a combinatorial problem. Of course, this does not mean that it becomes easy. In particular, determining if two arbitrary AF-algebras are stably isomorphic is undecidable \cite{GIBP}. However, in the case of stationary Bratteli diagrams, i.e. Bratteli diagrams with constant transition matrices, the problem is decidable \cite{DIP}. Note that while \cite{DIP} deals with the problem of stable isomorphism, two AF-algebras being isomorphic implies that that they are stably isomorphic so the results in that paper are relevant to us. We end this section with a well known consequence of Theorem \ref{BratteliAF}. 

\begin{proposition}\label{telescopeequiv}
Two Bratteli diagrams $(V,E,r,s,d)$, $(V',E',r',s',d')$ are telescope equivalent if and only if there exists a Bratteli diagram $(V'',E'',r'',s'',d'')$ such that telescoping on odd levels leads to a telescoping of one of $(V,E,r,s,d)$ or $(V',E',r',s',d')$ and telescoping on even levels leads to a telescoping of the other. 
\end{proposition}

\section{\textbf{Symbolic Substitutions and Bratteli Diagrams}}

Suppose $\sigma$ is a symbolic substitution on a finite set of letters $\mathcal L$. 
To each letter $l_{1},l_{2},...,l_{n}\in \mathcal L$ associate a vertex and call this set $V_{0}$. For each vertex in $v\in V_{0}$, let $d(v)=1$. To each of $\sigma (l_{1}),...,\sigma (l_{n})$ associate a vertex and call this set $V_{1}$. Let the transition matrix be equal to the substitution matrix corresponding to the substitution rule. Continue this process recursively. This will generate a labelled Bratteli diagram associated to the symbolic substitution. In this way, the labeling of the vertex corresponding to $\sigma^n(l_i)$, $n\in \mathbb{N}$ is equal to the length of the word and the number of edges between $\sigma^{n-1}(l_i)$ and $\sigma^{n}(l_j)$ is equal to the number of occurrences of $\sigma^{n-1}(l_i)$ in $\sigma^{n}(l_j)$.

Hence, all of the chain systems defined from this substitution are of the form:\\
\begin{figure}[H]
\begin{tikzcd}
\mathcal A_1=\displaystyle{\bigoplus_{i=1}^{n}}\mathbb{C} \arrow[r, "M"] & \mathcal A_2 \arrow[r, "M"]    & \mathcal A_3 \arrow[r, "M"]    & \mathcal A_4 \arrow[r, "M"]    & \cdots \\
\end{tikzcd}
\end{figure}
Where the $\mathcal A_j\cong \bigoplus_{i=1}^nM_{m_{(j,i)}}$ for some $j,i\in \mathbb{N}$ and $M$ is the substitution matrix. 


The study of symbolic substitutions and their associated Bratteli diagrams is often done in the context of dynamical systems, as seen in papers such as \cite{BIM} and \cite{AA}. The focus of this paper will be studying the substitution matrices associated to these Bratteli diagrams and looking at what remains invariant under telescope equivalence.

\begin{proposition}
Let $M$ and $N$ be substitution matrices on n letters $l_{1},...,l_{n}$. If there exist $p,q\in \mathbb{N}$ such that $M^p = N^q$, then $M$ and $N$ correspond to the same AF-algebra
\end{proposition}
This proposition follows immediately Proposition \ref{telescopeequiv}.


\begin{proposition}
Let $M$ and $N$ be substitution matrices on $n$ letters. Then they define isomorphic stationary Bratteli diagrams if and only if there exists a permutation matrix $P$ such that $M=P^{-1}NP$
\end{proposition}

\begin{proof}
Suppose $M$ and $N = P^{-1}MP$ are $n\times n$ substitution matrices for some permutation matrix $P$. This induces isomorphic Bratteli diagrams by considering $P$ as the bijection between vertices of the $n$th level. 

For the converse, Let $M$ and $N$ define isomorphic Bratteli diagrams. There exists a bijection between the vertices and edges which preserves the range, source and labels. Focusing on the $n$th level of vertices, this bijection can be represented by a permutation matrix $P_n$. Hence, $N = P_n^{-1}MP_{n+1}$ which then gives that all of the permutations are identical, or can be taken to be identical in the case when $M$ is singular.
\end{proof}

Of course, it is not the case in general that if two Bratteli diagrams are telescope equivalent and have transition matrices of the same size, that $M= P^{-1}N^{\frac{p}{q}}P$. Take for example:

$M=\begin{bmatrix}
3 & 3 \\ 3 & 3
\end{bmatrix}$
$N=\begin{bmatrix}
2 & 2 \\ 4 & 4
\end{bmatrix}$.
By noting that $A=\begin{bmatrix}
3 \\ 3
\end{bmatrix}\begin{bmatrix}
1 & 1
\end{bmatrix}$ and $B=\begin{bmatrix}
2 \\4
\end{bmatrix}\begin{bmatrix}
1 & 1
\end{bmatrix}$ and that $[6]=\begin{bmatrix}
1 & 1
\end{bmatrix}\begin{bmatrix}
3 \\ 3 
\end{bmatrix}= \begin{bmatrix}
1 & 1
\end{bmatrix}\begin{bmatrix}
2 \\ 4
\end{bmatrix}$ It is easy to see that both define Bratteli diagrams telescope equivalent to the stationary Bratteli diagram defined by the transition matrix $[6]$, cf. Lemma \ref{lem:makebigger} below. In the context of symbolic dynamics this equivalence is called \textit{elementary strong shift equivalence} \cite[Chapter 7]{LindMarcus}.

\begin{definition}
A Bratteli diagram is called {\bf\em simple} if for any $n\in \mathbb N$ and for any vertex $v\in V_n$, there exists $m\in \mathbb N, m>n$, such that there is a path between $v$ and all of the vertices in $V_m$.
\end{definition}

 In the case of a stationary Bratteli diagram, this implies that the transition matrix is primitive.

\begin{proposition}
Let $M$ be a primitive substitution matrix and let $(V,E,r,s,d)$ be the Bratteli diagram defined by $M$. Let $N$ be another substitution matrix such that its associated Bratteli diagram $(V',E',r',s',d')$ is telescope equivalent to $M$. Then $N$ is also primitive.
\end{proposition}
 
\begin{proof} 
It is clear from the definition that a Bratteli diagram being simple implies that its telescoping is also simple. This is because edges in the telescoped Bratteli diagram are made from the paths in the original Bratteli diagram. For the same reason, it is clear that if a Bratteli diagram is not simple, then it's telescoping is also not simple. From here the assertion follows.  
\end{proof}

Every stationary Bratteli diagram with a primitive $n\times n$ transition matrix $M$ will be telescope equivalent to a stationary Bratteli diagram defined by $m\times m$  primitive matrix for any $m>n$. To prove this we first need the following lemma. In Lind and Marcus, this is called \textit{state splitting} \cite[Section 2.4]{LindMarcus}.

\begin{lemma}\label{lem:makebigger}
Let $M$ be an $n\times n$ substitution matrix such that $M=NS$, were $N$ is an $n\times m$ non-negative integer matrix and $S$ is an $m\times n$ non-negative integer matrix. Further assume that the columns of $N$ sum up to one, and none of the rows of $N$ sum up to zero. Then the Bratteli diagram corresponding to $M$ is telescope equivalent to the Bratteli diagram corresponding to $SN$.
\end{lemma}

\begin{proof}
By how $N$ was defined, it is a valid transition matrix between $ \displaystyle{\bigoplus_{i=1}^{n}}\mathbb{C}$ and $ \displaystyle{\bigoplus_{i=1}^{m}}\mathbb{C}$. Therefore the chain system corresponding to $M$ is equal to: 

\[
\begin{tikzcd}
A_{1}=\displaystyle{\bigoplus_{i=1}^{n}}\mathbb{C} \arrow[r, "NS"] & A_{2} \arrow[r,"NS"] & A_{3} \arrow[r,"NS"] & A_{4} \arrow[r,"NS"] & \cdots \\
A_{1}=\displaystyle{\bigoplus_{i=1}^{n}}\mathbb{C} \arrow[r, "N"] & A_{2}'=\displaystyle{\bigoplus_{i=1}^{m}}\mathbb{C} \arrow[r, "S"] & A_{2} \arrow[r,"N"] & A_3' \arrow[r,"S"] & \cdots \\
A_{1}=\displaystyle{\bigoplus_{i=1}^{n}}\mathbb{C} \arrow[r, "N"] & A_{2}'=\displaystyle{\bigoplus_{i=1}^{m}}\mathbb{C} \arrow[r, "SN"] & A_{3}' \arrow[r, "SN"] & A_{4}' \arrow[r] & \cdots
\end{tikzcd}
\]

 which after telescoping off $A_{1},$ we get the stationary Bratteli diagram defined by $SN$.
\end{proof}

\begin{proposition}\label{prop:equivtolarger}

Let $M$ be a primitive substitution matrix on $n$ letters and let $m>n$ be another natural number. Then there exists a primitive substitution matrix $D$ on m letters such that the Bratteli diagram defined by $N$ is telescope equivalent to the Bratteli diagram defined by $D$. 
\end{proposition}

\begin{proof}
Let $M$=$[m_{ij}]$ be a primitive substitution matrix. Assume that all of the entries of $M$ are greater than or equal to $m$  (if not take a suitable power of $M$). Now define
\[
N = \left[\begin{matrix} 1 & \cdots & 1 & 0_{1,n-1} \\ 0_{n-1,1}& \cdots & 0_{n-1, 1} & I_{n-1} \end{matrix}\right]
\]
\[
S = \left[\begin{matrix} m_{11} - (m-n+1) & \cdots & m_{1n} - (m-n+1) \\
I_{m-n+1,1} & \cdots & I_{m-n+1, 1}
\\ m_{21} & \cdots & m_{2n} 
\\ \vdots && \vdots
\\ m_{n1} & \cdots & m_{nn}
\end{matrix}\right]
\]


Let $D=SN$. Then this will be a valid primitive substitution matrix on $m$ letters as all rows/columns have non-zero entries. Thus by the above lemma, the Bratteli diagrams defined by $M$ and $D$ are telescope equivalent
\end{proof}

To illustrate the proposition consider the following:

\begin{example}

\[
\begin{bmatrix}
1 & 1 \\  1 & 0
\end{bmatrix}^{5}
= \begin{bmatrix}
8 & 5 \\  5 & 3
\end{bmatrix}
= \begin{bmatrix}
1 & 1 & 0 \\  0 & 0 & 1
\end{bmatrix} \begin{bmatrix}
7 & 4  \\  1 & 1 \\ 5 & 3
\end{bmatrix} \quad \textrm{and} \quad
\begin{bmatrix}
7 & 4  \\  1 & 1 \\ 5 & 3
\end{bmatrix} \begin{bmatrix}
1 & 1 & 0 \\  0 & 0 & 1
\end{bmatrix} = \begin{bmatrix}
7 & 7 & 4 \\ 1 & 1 & 1 \\ 5 & 5 & 3
\end{bmatrix}
\]
By the above proposition both of these substitution matrices on 2 and 3 letters respectively define telescope equivalent Bratteli diagrams and thus $\ast$-isomorphic AF-algebras.
\end{example}
\section{\textbf{Bratteli Diagrams and transition matrix rank}}

A natural question one can ask is: given a non-invertible substitution matrix $M$ can we find an invertible substitution matrix $N$ on fewer letters such that the Bratteli diagrams that they define are telescope equivalent. The answer is no. Before we can provide a simple counterexample, we shall need a few things:

\begin{definition}
A \textbf{\textit{uniformly hyperfinite (UHF)}} C*-algebra is a unital C*-algebra for which there exists a chain of unital subalgebras $A_{1}\cong M_{n_1}\subseteq A_{2}\cong M_{n_2}\subseteq \cdots$ dense inside the UHF-algebra were $M_{n_k}$ denotes a full matrix algebra and for all $k\in \mathbb{N}$ $n_{k}|n_{k+1}$. From this definition, it is clear that UHF-algebras are a special type of AF-algebra
\end{definition}

The sequence of positive integers such that $n_{k}|n_{k+1}$ for all k determines a formal product $n=d_{1}d_{2}...,$ where $n_{k}=d_{1}d_{2}...d_{k}$ for all k. This formal product can be written as $n=2^{r_{1}}3^{r_{2}}...$ where $0\leq r_{k} \leq + \infty$ for all k. n is called a \textbf{\textit{generalized integer}} or \textbf{\textit{supernatural number}}. The sequence of full matrix algebras in a UHF-algebra determines a supernatural number. Moreover, two UHF-algebras are $\ast$-isomorphic if and only if they have the same supernatural number. For a review of these concepts see \cite{LMTA}.  

If the Bratteli diagram has a $1\times 1$ transition matrix or is telescope equivalent to one that does, then the resulting AF-algebra is a UHF-algebra. As a result, all symbolic substitutions on 1 letter will correspond to a UHF-algebra whose supernatural number is of the form $p_{1}^{\infty}p_{2}^{\infty}...p_{n}^{\infty}$ for some primes $p_i$.
Obviously any UHF-algebra can be made to start at $\mathbb{C}$, but if they are not of the above form, they will no longer be stationary.
\begin{example}
Consider $M = \begin{bmatrix}
 1 & 2 \\ 1 & 2
 \end{bmatrix} = \begin{bmatrix}1 \\ 1\end{bmatrix}\begin{bmatrix}
 1 & 2
 \end{bmatrix} = CD$ and the following chain system
 \[
 \mathbb C \xrightarrow{C^T}\mathbb{C}\oplus \mathbb{C} \xrightarrow{C} M_{2} \xrightarrow{D} M_{2} \oplus M_{4} \xrightarrow{C} M_{6} \xrightarrow{D} M_{6} \oplus M_{12} \rightarrow ...
 \]
 Telescoping to the odd terms gives the Bratteli diagram associated to $M$ while telescoping on the even terms corresponds to the  $2\cdot 3^{\infty}$ UHF-algebra. By the above explanation we conclude that the Bratteli diagram associated to $M$ cannot be telescope equivalent to the Bratteli diagram associated to a substitution on 1 letter since it is not of the form $p_{1}^{\infty}p_{2}^{\infty}...p_{n}^{\infty}$.
\end{example}
This allows us now to show that a matrix and its transpose do not in general define $*$-isomorphic AF-algebras.
\begin{example} The transpose $M^T = \begin{bmatrix}
 1 & 1 \\ 2 & 2
 \end{bmatrix} = \begin{bmatrix}1 \\ 2\end{bmatrix}\begin{bmatrix}
 1 & 1
 \end{bmatrix} = D^TC^T$ corresponds to the $3^\infty$ UHF-algebra as indicated by the chain system
 \[
 \mathbb C \xrightarrow{C^T}\mathbb{C}\oplus \mathbb{C} \xrightarrow{D^T} M_{3} \xrightarrow{C^T} M_{3} \oplus M_{3} \xrightarrow{D^T} M_{9} \xrightarrow{C^T} M_{9} \oplus M_{9} \rightarrow ...
 \]
 Therefore, the transpose of a substitution matrix need not lead to a telescope equivalent Bratteli diagram.

\end{example}

\begin{proposition}
Let $M$ and $N$ be $n\times n$ substitution matrices that define  telescope equivalent Bratteli diagrams. Then $M$ is invertible if and only if $N$ is invertible. 
\end{proposition}
\begin{proof}
By Proposition \ref{telescopeequiv} there exists a Bratteli diagram $(V,E,r,s,d)$ such that, without loss of generality, telescoping on even levels is a telescoping of the Bratteli diagram for $M$ and telescoping on odd levels is a telescoping of the Bratteli diagram for $N$. 
Hence, if the chain system for $(V,E,r,s,d)$ is
\[
\mathcal A_1 \xrightarrow{C_1} \mathcal A_2 \xrightarrow{C_2} \mathcal A_3 \xrightarrow{C_3} \cdots
\]
then
\[
\mathcal A_1 \xrightarrow{C_1C_2 = M^{p_1}} \mathcal A_3 \xrightarrow{C_3C_4 = M^{p_2}} \mathcal A_5 \xrightarrow{C_5C_6 = M^{p_3}} \ \cdots
\]
and
\[
\mathcal A_2 \xrightarrow{C_2C_3 = N^{q_1}} \mathcal A_4 \xrightarrow{C_4C_5 = N^{q_2}} \mathcal A_6 \xrightarrow{C_6C_7 = N^{q_3}} \ \cdots
\]
Thus $M^{p_1+p_2} = C_{1}C_{2}C_{3}C_{4}=C_{1}N^{q_{1}}C_{4}$ and $N^{q_1+q_2} = C_{2}C_{3}C_{4}C_{5}=C_{2}M^{p_{2}}C_{5}$ which implies the desired result.
\end{proof}


In general, suppose $M$ and $N$ are $n\times n$ non-negative integer-valued matrices that have telescope equivalent Bratteli diagrams. Note that $M^n$ and $N^n$ still maintain telescope equivalence and no longer have nilpotent Jordan blocks. Thus, by the proof of the last proposition it is immediate that $rank(M^n) = rank(N^n)$.
This then implies that telescope equivalence is not rank preserving but instead preserves the rank of the non-nilpotent part of the substitution matrix.

The following theory is classic and found in various forms in many sources, e.g. \cite{NSB}, but we prove it here as it is fundamental to understanding the telescope equivalence of symbolic substitutions.

\begin{proposition}
Let $M,N$ be two invertable $n\times n$ matrices. Then they define telescope equivalent Bratteli diagrams if and only if there exists an invertible $n\times n$ matrix $J$ with non-negative integer entries, a non-negative integer $m$  and strictly increasing sequences of positive integers $\{k_n\},\{l_n\}$ such that 
\[ J^{-1}M^{k_1} , \ \ M^{-k_n}JN^{l_{n}}, \ \  \textrm{and} \ \ N^{-l_n}J^{-1}M^{k_{n+1}}
\]
are non-negative integer matrices, and
\[
J\left[\begin{matrix} 1 \\ \vdots \\ 1\end{matrix}\right] = N^m\left[\begin{matrix} 1 \\ \vdots \\ 1\end{matrix}\right]
\]
\end{proposition}
\begin{proof}
Assume the Bratteli diagrams are telescope equivalent and suppose we have the same chain system $(\mathcal A_i, C_i)$ as in the previous proof. Define $\mathcal A_0 = \bigoplus_{i=1}^n \mathbb C$. There must exist an integer $d\geq 0$ such that $M^d$ is a transition matrix from $\mathcal A_0$ to $\mathcal A_1$. 
Let $J=M^dC_1$, $k_n = d+ p_1+ \cdots + p_n$ and $l_n = q_1\cdots + \cdots + q_n$.
Then
\[
M^{k_n}C_{2n+1} \ = \ M^dM^{p_1}\cdots M_{p_n}C_{2n+1} \ = \  M^dC_1N^{q_1}\cdots N_{q_n} = JN^{l_n} 
\]
which gives $M^{-k_n}JN^{l_n} = C_{2n+1}.$
Similarly, 
\begin{align*}
M^{k_{n+1}} \ = \ M^dM^{p_1}\cdots M^{p_{n+1}}\ &= \ M^dC_1C_2\cdots C_{2n+1}C_{2n+2}  
\\ &= \ JN^{q_1}\cdots N^{q_{n}}C_{2n+2}\  =\  JN^{l_n}C_{2n+2}
\end{align*}
which gives $J^{-1}M^{k_1}$ and 
$N^{-l_n}J^{-1}M^{k_{n+1}} = C_{2n+2}.$ Lastly, note that $\mathcal A_2$ is a finite-dimensional C$^*$-algebra in the chain system for the $N$ matrix substitution. Hence, there exists a non-negative integer $m$ such that $N^m$ is a transition matrix from $\mathcal A_0$ to $\mathcal A_2$. But so is $J = M^dC_1$ which gives the required condition.

Conversely, suppose $J^{-1}M^{k_1}$, $M^{-k_n}JN^{l_{n}}$, and $N^{-l_n}J^{-1}M^{k_{n+1}}$ are non-negative integer matrices for $J$, $\{k_n\}$ and $\{l_n\}$ as specified and that \[
J\left[\begin{matrix} 1 \\ \vdots \\ 1\end{matrix}\right] = N^m\left[\begin{matrix} 1 \\ \vdots \\ 1\end{matrix}\right]
\]
for some $m\geq 0$. Let $C_1 = J$, $C_2 = J^{-1}M^{k_1}$ and for every $n\in \NN$ define
\[
C_{2n+1} = M^{-k_n}JN^{l_{n}} \ \ \textrm{and} \ \ C_{2n+2} = N^{-l_n}J^{-1}M^{k_{n+1}}.
\]
This implies that
\[
C_{2n-1}C_{2n} = M^{k_n} \ \ \textrm{and} \ \ C_{2n}C_{2n+1} = N^{l_n}.
\]
Define the chain system $(\mathcal A_i, C_i)$ where $\mathcal A_1 = \bigoplus_{i=1}^n \mathbb C$ and the rest of the algebras are uniquely implicitly defined by the $C_i$. 
\end{proof}

\begin{example} In \cite{NSB} it was shown that the following two invertible substitution matrices define telescope equivalent Bratteli diagrams:
\[
M=\begin{bmatrix}
1 & 1 & 0 & 0 & 0 \\
0 & 1 & 1 & 0 & 0 \\
0 & 0 & 1 & 1 & 0 \\
0 & 0 & 0 & 1 & 1 \\
1 & 0 & 0 & 0 & 1 
\end{bmatrix} \quad \textrm{and} \quad 
N=\begin{bmatrix}
0 & 1 & 1 & 0 & 0 \\
0 & 0 & 1 & 1 & 0 \\
1 & 0 & 0 & 0 & 1 \\
1 & 1 & 0 & 0 & 0 \\
0 & 0 & 0 & 1 & 1 
\end{bmatrix}
\]
Do note that $M$ is not equal to a permutation conjugation applied to a rational power of $N$, or vice-versa.

\end{example}


Recall from Theorem \ref{thm:aperiodic}, that a substitution matrix $M$ having an irrational PF-eigenvalue implies that any substitution rule corresponding to that matrix is aperiodic. As it turns out, Bratteli diagrams, and thus the AF-algebras are able to tell the difference between these purely aperiodic substitution matrices and those that are not.  

\begin{proposition}
Bratteli diagrams and thus AF-algebras are able to tell the difference between substitution matrices which are purely aperiodic and those which are not. 
\end{proposition}

\begin{proof}
In \cite{DIP}, it was shown that if $M$ and $N$ are two transition matrices defining telescope equivalent Bratteli diagrams, then the extension field $\mathbb{Q}[\lambda_{M}]$ is equal to $\mathbb{Q}[\lambda_{N}]$ were $\lambda_M$ and $\lambda_N$ are the PF-eigenvalues of $M$ and $N$ respectively. If $\lambda_{M}$ is irrational and $\lambda_{N}$ rational, then $\mathbb{Q}[\lambda_{M}]\neq \mathbb{Q}[\lambda_{N}]$, a contradiction. Thus it is the case that stationary Bratteli diagrams defined by purely aperiodic substitution matrices are only telescope equivalent to Bratteli diagrams also defined by purely aperiodic substitution matrices.
\end{proof}

\section{\textbf{Ordered Bratteli diagrams and substitution tilings}}
A disadvantage of Bratteli diagrams in encoding symbolic substitutions is that they ignore much of the substitution's structure. For example, all substitutions with the same substitution matrix end up defining the same Bratteli diagram. One way of making this more selective is to use ordered Bratteli diagrams instead. These define a partial order on the edges, which for us will be induced by the substitution rule.

\begin{definition}
An {\bf\em ordered Bratteli diagram} $(V,E, r,s,d, \leq)$ is a Bratteli diagram $(V,E,r,s,d)$ with a partial order on the edges such that two edges $e,e'$ are comparable if and only if $r(e)=r(e')$. 
\end{definition}

This definition naturally extends to a partial order on the set of all finite paths where two paths $p_1=(e_k,e_{k+1},...,e_{k+l}), \, p_2=(e'_k,e'_{k+1},...,e'_{k+n})$ are comparable if and only if they are between the same levels of vertices and $r(e_{k+n})=r(e'_{k+n})$. In this case, we say that $p_1<p_2$ if and only for some $i$ with $k+1 \leq i\leq k+n$, we have
\[
e_i < e'_i \quad \textrm{and} \quad e_j=e'_j, \ i < j\leq k+n. 
\]

A {\bf\em telescoping} of an ordered Bratteli diagram $(V,E,r,s,d, \leq)$ is a telescoping of the Bratteli diagram $(V,E,r,s,d)$ where the order is induced by the above path comparison. The definition of telescope equivalent ordered Bratteli diagrams follows immediately and we echo Proposition \ref{telescopeequiv} in this context.

\begin{proposition}
Two ordered Bratteli diagrams $(V,E,r,s,d, \leq)$, \\$(V',E',r',s',d', \leq')$ are telescope equivalent if and only if the exists an ordered Bratteli diagram such that $(V'',E'',r'',s'',d'', \leq '')$ such that telescoping on odd levels lead to a telescoping of $(V,E,r,s,d, \leq)$ and telescoping on even levels leads to a telescoping of $(V',E',r',s',d', \leq ')$
\end{proposition}

The partial order allows us to consider maximal and minimal edges and paths. Indeed, for any pair of path-connected vertices $v, v'\in V$ there is a unique maximal path and a unique minimal path between $v$ and $v'$, with the possibility that they are the same. This can be extended as follows.



\begin{definition}
Let $(V,E,r,s,d, \leq)$ be an ordered Bratteli diagram. Let $X$ denote the set of all {\bf\em infinite paths} that start from the first level of vertices. We say that an infinite path is {\bf \em maximal} if each edge in the path is maximal. Similarly, we say that an infinite path is {\bf\em minimal} if each edge in the path is minimal. We denote the set of all maximal and minimal infinite paths as $X_{max}$ and $X_{min}$ respectively. 
\end{definition}

From the existence of maximal and minimal paths between any two path-connected vertices mentioned above it is a simple matter to show that $X_{max}$ and $X_{min}$ are always nonempty.



\begin{proposition}\label{minmaxtelescope}
Let $(V,E, r,s,d, \leq)$ be an ordered Bratteli diagram with a telescoping $(V',E',r',s',d' \leq ')$. Then $(V,E, r,s,d, \leq)$ and\\ $(V',E',r',s',d',\leq')$ have the same number of maximal infinite paths and the same number of minimal infinite paths.
\end{proposition}

\begin{proof}
Let $0<n_1<n_2<...$ define the telescoping of $(V',E',r',s',d'\leq ')$. Define a bijective map $F:X_{(V,E)}\rightarrow X_{(V',E')}$ on the infinite paths by
\[
F(e_1,e_2,\dots) = ((e_1,\dots, e_{n_1}), (e_{n_1+1}, \dots, e_{n_2}), \dots).
\]
Suppose $(e_1,e_2,...)$ is maximal then each $e_i$ is maximal by definition. Hence, each path $(e_{n_j+1},\dots, e_{n_{j+1}})$ is also maximal which implies that $F(e_1,e_2,\dots)$ is maximal. The minimal path argument follows identically.
\end{proof}

\begin{corollary}\label{telescopemaxmin}
Let $(V,E,r,s,d,\leq)$, $(V',E',r',s',d,', \leq')$ be two telescope equivalent stationary ordered Bratteli diagrams. Then they have the same number of maximal infinite paths and the same number of minimal infinite paths.
\end{corollary}

\begin{proposition}
Let $(V,E,r,s,d)$ be a simple Bratteli diagram such that $|E_n|>1$ for infinitely many $n$. Then $X_{min}$ and $X_{max}$ are disjoint.
\end{proposition}
\begin{proof}
Since $|E_n|>1$ infinitely often, we can choose an increasing sequence $(n_k)_{n_k \in\mathbb{N}}$ so that $|E_{n_k}|>1$ for all $n_k$. Since $(V,E,r,s,d)$ is simple, there exists an increasing subsequence $(n_{k_{l}})_{n_{k_{l}}\in\mathbb{N}}$ so that there exists an edge between $v$ and $v'$ for all $v\in V_{n_{k_{l}}}$, $v'\in V_{n_{k_{l+1}}}$. This guaranties that every entry in the transition matrices at each level is at least 1. Finally, telescoping one last time on every second level, we get that every entry in the transition matrix  is strictly greater than 1. In this telescoping there are at least two edges between any two vertices on adjacent levels, which results in the set of maximal and minimal edges, being disjoint. Therefore, by definition this implies that infinite maximal and minimal paths are also disjoint. If we assume for contradiction that $(V,E,r,s,d)$ had a path that was both maximal and minimal, then as seen in the proof of Proposition \ref{minmaxtelescope}, the telescoping of that path would also be both maximal and minimal, a contradiction. Therefore, the maximal and minimal paths in $(V,E,r,s,d)$ are disjoint.
\end{proof}

In particular, the above proposition implies that the maximal and minimal infinite paths are disjoint for every ordered Bratteli diagram arising from a (nontrivial) symbolic substitution with primitive substitution matrix.

One can also go on to define a properly ordered Bratteli diagram as a simple ordered Bratteli diagram where there is a unique maximal and minimal infinite path. 
Properly ordered Bratteli diagrams correspond to a type of symbolic substitution called a proper substitution. These are substitutions where for every letter $l\in \mathcal L$ and every $n\in \mathbb{N}$, there exist letters $u,v \in \mathcal L$ such that $u$ is the first letter of $\sigma^n(l)$ and $v$ is the last. Proper substitutions have been considered in papers such as \cite{AA}.

To define an ordered Bratteli diagram from a symbolic substitution, the procedure is the same as before, except now the partial order on the edges is induced by the order that the letters appear in the substitution rule. We will label comparable edges with numbers to indicate the order.

\begin{example}\label{OBDSubExample}
Consider the periodic symbolic substitution $\sigma$ on the letters $\{a,b\}$ given by
\begin{align*}
    \sigma(a)&=ab \\
    \sigma(b)&=ba
\end{align*}
This corresponds to the following ordered Bratteli diagram:
\[
\centerline{\includegraphics{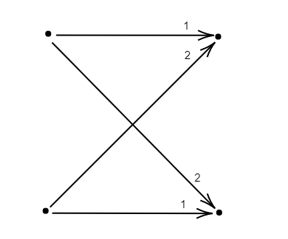}}
\]
\end{example}

As in Section 3, we encode the diagrams into an operator algebra direct limit.

\begin{definition}
Suppose $\mathcal{A} =\overline{\bigcup\limits_{n=1}^{\infty}A_{n}}$ is an AF-algebra. A norm-closed subalgebra $\mathcal T \subseteq \mathcal A$ is called a {\bf\em triangular AF-algebra (TAF-algebra)} if $\mathcal T \cap \mathcal T^* \cap A_n$ is a maximal abelian self-adjoint subalgebra (MASA) of $A_n$, a diagonal, for every $n$. If $\mathcal A$ is a UHF-algebra then $\mathcal T$ is called a TUHF-algebra. From this we see that $\mathcal T$ is the inductive limit of triangular subalgebras of finite-dimensional C$^*$-algebras.
\end{definition}

There is a rich history of research in TAF-algebras and their structure is quite complicated without further assumptions, see \cite{LMTA}. For instance, there are infinitely many not isometrically isomorphic TAF-subalgebras of a given AF-algebra, even in the case where the AF-algebra is simple. We will be focusing our attention on a particularly nice class of TAF-algebras.

\begin{definition}
Suppose $\mathcal T = \varinjlim \mathcal B_{n}$ is a TAF-subalgebra of $\mathcal A =\varinjlim \mathcal A_{n}$. We call $\mathcal T$ a {\bf\em standard TAF-algebra} if for every $n$
\[
\mathcal A_n = \bigoplus_{i=1}^{k_n} M_{m_{i,n}} \quad \textrm{and} \quad  \mathcal B_n = \bigoplus_{i=1}^{k_n} T_{m_{i,n}},
\]
where $T_j$ is the $j\times j$ upper triangular matrices,
and the embeddings $\mathcal B_n \hookrightarrow \mathcal B_{n+1}$ are standard
\[
\bigoplus_{i=1}^{k_n} T_{m_{i,n}} \ \ni \ a_1\oplus\cdots\oplus a_{k_n}  \quad \mapsto \quad
\bigoplus_{j=1}^{k_{n+1}} \left( \oplus_{l=1}^{p_j} \ a_{\rho(j,l)} \right) \ \in \ \bigoplus_{i=1}^{k_{n+1}} T_{m_{i,n+1}}
\]
with $\rho(j,l) \in \{1,\dots, k_n\}$ such that the sizes of the matrices add up correctly.
\end{definition}

This class has been well-studied, standard TAF-algebras are strongly maximal, have $*$-extendable regular embeddings, and support an integer-valued cocycle. Moreover, they are isometrically isomorphic to tensor algebras of C$^*$-correspondences \cite{LimitAlg}, a very special structure property that few TAF-algebras share.
Poon and Wagner \cite{ZTAF} called these standard $\mathbb Z$-analytic TAF-algebras and proved that they completely characterize ordered Bratteli diagrams, echoing the work of Bratteli in Theorem \ref{BratteliAF}, as we will see below.

First, however, we need to consider how to encode an ordered Bratteli diagram $(V,E,r,s,d, \leq)$ into a standard TAF-algebra $\mathcal T = \varinjlim \mathcal B_n$. If $V_n = \{v_1,\cdots, v_{k_n}\}$ then define $\mathcal B_n = \bigoplus_{i=1}^{k_n} T_{d(v_i)}$. Now for $v\in V_{n+1}$ we have $r^{-1}(v) = \{e_1,\dots, e_m\}$ where $e_1 < e_2 < \cdots < e_m$. Thus, define the mapping of $\mathcal B_n$ into $T_{d(v)}$ as 
\[
\bigoplus_{i=1}^{k_n} T_{d(v_i)} \ni a_{v_1} \oplus \cdots \oplus a_{v_{k_n}} \
\mapsto \ a_{s(e_1)}\oplus a_{s(e_2)} \oplus \cdots \oplus a_{s(e_m)} \in T_{d(v)}\subseteq \mathcal B_{n+1}
\]

\begin{example}
Consider the ordered Bratteli diagram arising from a symbolic substitution in Example \ref{OBDSubExample}. We can associate to this the standard TAF-algebra $\varinjlim T_{2^n} \oplus T_{2^n}$ by 
\[
T_{2^n} \oplus T_{2^n} \ \ni\ a \oplus b \ \ \mapsto \ \ (a\oplus b) \oplus (b\oplus a) \ \in \ T_{2^{n+1}} \oplus T_{2^{n+1}}.
\]
\end{example}

This shows how natural it is to take a symbolic substitution and convert it into a standard TAF-algebra. 

\begin{theorem}[Theorem 3.7 \cite{ZTAF}]
Two ordered Bratteli diagrams are telescope equivalent if and only if their associated standard TAF-algebras are isometrically isomorphic.
\end{theorem}

\begin{remark}
There is a strong link between standard TAF-algebras and partial dynamical systems. If one defines a partial homeomorphism on the path space $X\backslash X_{max}$ into $X\backslash X_{min}$, called the Vershik map (essentially the successor map on finite paths extended to the infinite path space), then one gets a partial dynamical system. It is shown in \cite[Theorem 4.16]{TAF} that the partial dynamical systems defined by two ordered Bratteli diagrams are conjugate if and only if the standard TAF-algebras are isometrically isomorphic. When an ordered Bratteli diagram is proper, the Vershik map can be extended to the entire path space. The resulting dynamical system is called a Cantor minimal system. 
\end{remark}




Corollary \ref{telescopemaxmin} provides a useful, but weak method of telling if two stationary ordered Bratteli diagrams are telescope equivalent. If we consider the following ordered Bratteli diagrams with transition matrix $M=\begin{bmatrix} 1 & 1 \\ 1 & 0
\end{bmatrix}$, with maximal paths coloured red, and minimal paths coloured green, we can instantly tell that they are not equivalent as the number of maximal and minimal paths differ.\\
\[
\centerline{\includegraphics{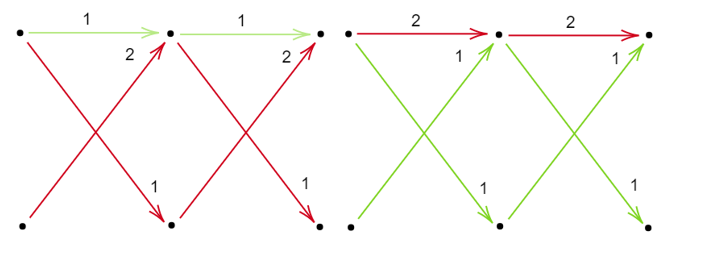}}
\]

Therefore, the two symbolic substitutions below have standard TAF-algebras that are not isometrically isomorphic.
\[
\begin{array}{l} \sigma_1(a) = ab \\ \sigma_1(b) = a \end{array}
\quad \quad \textrm{and} \quad \quad
\begin{array}{l} \sigma_2(a) = ba \\ \sigma_2(b) = a \end{array}.
\]


The number of maximal (minimal) paths may be identical across two different orders on the same substitution, precluding the argument above. However, we may still be able to tell these substitutions apart by way of their ordered Bratteli diagrams.

\begin{example}
Consider the following two orders of the substitution matrix \[
\left[\begin{matrix} 2 & 2 \\ 2 & 2 \end{matrix}\right] = \left[\begin{matrix}  1 \\  1 \end{matrix}\right] \left[\begin{matrix} 2 & 2 \end{matrix}\right] 
\]
which is telescope equivalent in the unordered sense to the substitution matrix $[4]$ on 1 letter.
\[\centerline{\includegraphics[scale=1]{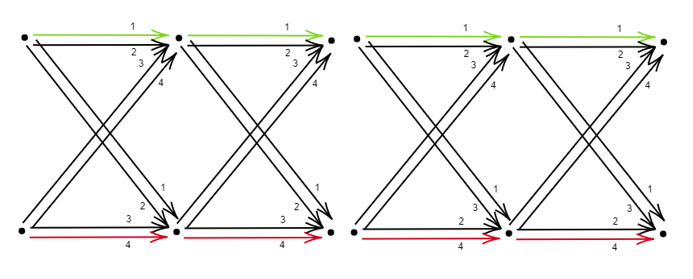}}
\]
While they both have a unique maximal and minimal path, they are not telescope equivalent as ordered Bratteli diagrams, see the end of section 6 in \cite{HPS}.
\end{example}


The previous examples show that the ordered Bratteli diagrams do tell some substitutions apart but we will see that they cannot distinguish among an infinite family.

\begin{example}
Consider the symbolic substitutions on two letters:
\[
\begin{array}{l} \sigma_1(a) = aaab \\ \sigma_1(b) = aaab \end{array}
\quad \quad \textrm{and} \quad \quad
\begin{array}{l} \sigma_2(a) = aabb \\ \sigma_2(b) = aabb \end{array}.
\]
These two substitutions clearly result in distinct tilings but have telescope equivalent ordered Bratteli diagrams as pictured below: 

\[
\centerline{\includegraphics[scale=0.75]{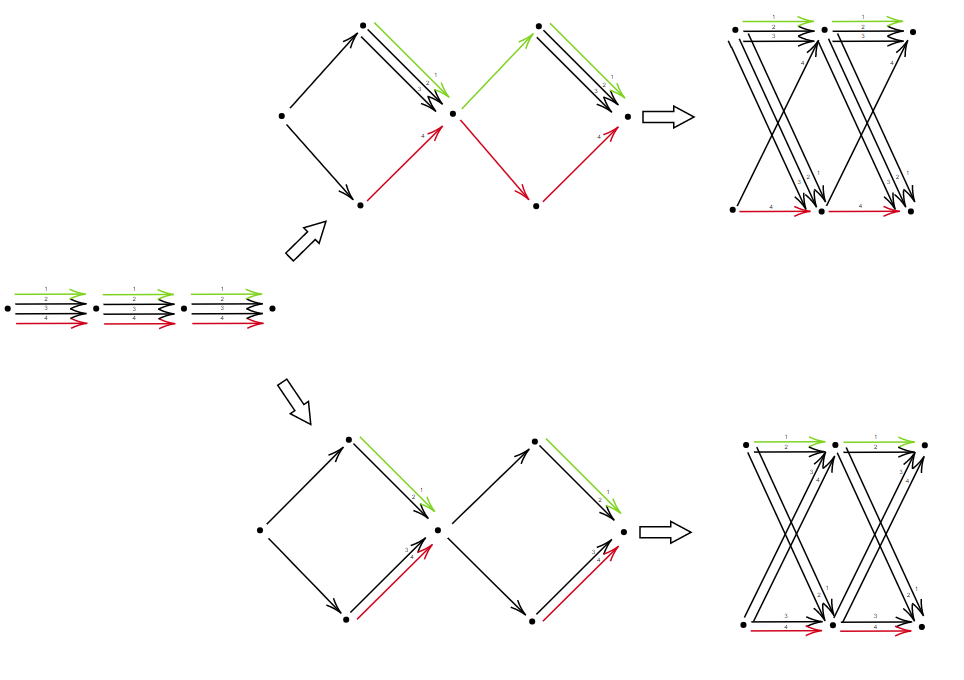}}
\]

Moreover, using the same ideas one can see that for any $k,n\in\mathbb N$ such that $1\leq k\leq n$ we have that the symbolic substitution on letters $\{a_1,\dots, a_k\}$ given by $\sigma$ such that
\[
\sigma(a_1) = \cdots = \sigma(a_k),
\]
where $\sigma(a_1)$ is a word of length $n$,
is telescope equivalent to the symbolic substitution $\sigma'$ on one letter $\{a\}$ such that
\[
\sigma'(a) = \underbrace{a\cdots a}_{n}.
\]
Therefore, two such substitutions have telescope equivalent ordered Bratteli diagrams if and only if their lengths have the same set of prime divisors (since the standard TUHF-algebras are classified by their supernatural numbers).
\end{example}

These ideas can be applied to any substitution but we can also show this collapse of an infinite set of symbolic substitutions another way.

\begin{proposition}
Every symbolic substitution has an ordered Bratteli diagram that is telescope equivalent to the ordered Bratteli diagrams of an infinite number of distinct symbolic substitutions.
\end{proposition}
\begin{proof}
Suppose $\sigma$ is a symbolic substitution on $n$ letters. For every $m>n$, Proposition \ref{prop:equivtolarger} gives a way to construct a symbolic substitution on $m$ letters whose (unordered) Bratteli diagram is telescope equivalent to that of $\sigma$. It is a simple matter to induce an ordering on this new Bratteli diagram that is then telescope equivalent to the original ordered Bratteli diagram.
\end{proof}



We end this paper with a discussion of the possible substitution tilings associated to powers of the Fibonacci substitution matrix, $F = \left[\begin{matrix}1&1\\1&0\end{matrix}\right]$. Because of a notion of ``prime'' matrices we can explicitly say when two such tilings have telescope equivalent ordered Bratteli diagrams.

\begin{proposition}
Suppose there are $2\times 2$ non-negative integer matrices $A,B$ such that $AB = F^m$ for some $m\in \NN$. Then there exists $k,l\in\NN\cup\{0\}$ such that $k+l = m$ and either 
\[
A = F^k, \ B=F^l \quad \textrm{or} \quad A=F^kJ, \ B= JF^l
\]
where $J=\left[\begin{matrix}0&1\\1&0\end{matrix}\right]$.
\end{proposition}
\begin{proof}
In \cite{PM} it is shown that every $2\times 2$ non-negative integer matrix with determinant 1 has a unique factorization into products of 
\[
P = \left[\begin{matrix}1&0\\1&1\end{matrix}\right] = JF \quad \textrm{and} \quad Q = \left[\begin{matrix}1&1\\0&1\end{matrix}\right] = FJ,
\]
with the usual convention that $I = P^0Q^0$. These two matrices then act as primes in this context.

Now suppose $AB = F^m, m\in \NN$, with $A,B$ $2\times 2$ non-negative integer matrices. This implies that both $|\det(A)| = |\det(B)| = 1$.

\vskip 12 pt
\noindent {Case 1: $m=2m'$:} Then
\[
AB = F^m = (FJJF)^{m'} = (QP)^{m'}.
\]
If $\det(A) = \det(B) = 1$ then by the unique factorization theorem \cite{PM} we find $A$ and $B$ simply by splitting $(QP)^{m'}$ into two parts (deconcatenation).
So either 
\[
A = (QP)^{k'} = F^{2k'}, \ B = (QP)^{l'} = F^{2l'}, \ 2k' + 2l' = 2m' = m
\]
or
\[
A=(QP)^{k'}Q = F^{2k'+1}J, \ B = P(QP)^{l'} = JF^{2l'+1}, \ 2k' + 2l' + 2 = 2m' = m
\]
On the other hand, if $\det(A) = \det(B) = -1$ then we can repeat the above argument with 
$(AJ)(JB) = F^m = (QP)^{m'}$ which again yields the desired conclusion.

\vskip 12 pt
\noindent {Case 2: $m = 2m'+1$:} Then
\[
JAB = JF^m = JF(FJJF)^{m'} = P(QP)^{m'}.
\]
Using the same argument as in the even case, via deconcatenation, yields the desired result.
\end{proof}

\begin{theorem}
Suppose $\sigma_M$ and $\sigma_N$ are symbolic substitutions with associated matrices $M = F^m$ and $N=F^n, m,n\in\NN$.
Then they have telescope equivalent ordered Bratteli diagrams if and only if there exist symbolic substitutions $\sigma_i$ with substitution matrices $F^{l_i}$ and positive integers $k, p_i,q_i, i\in\NN$, such that
\[
l_1 = kn, \ \ \sigma_{2i}\circ\sigma_{2i-1} = \sigma_M^{p_i} \ \ \textrm{and} \ \ \sigma_{2i+1}\circ\sigma_{2i} = \sigma_N^{q_i}
\]
\end{theorem}
\begin{proof}
The backward direction is immediate as the condition just says that they have equal telescopings as ordered Bratteli diagrams.

For the forward direction, assume that $M$ and $N$ lead to telescope equivalent ordered Bratteli diagrams. This section this implies that there are symbolic substitutions with  matrices $C_i,i\in \NN$ and $p_i,q_i\in\NN, i\in\NN$, such that
\[
C_{2i-1}C_{2i} = M^{p_i} \quad \textrm{and} \quad C_{2i}C_{2i+1} = N^{q_i}.
\]
Ignoring the orders we see that
\[
C_{2i-1}C_{2i} = F^{mp_i} \quad \textrm{and} \quad C_{2i}C_{2i+1} = F^{nq_i}.
\]
The previous proposition then gives that there exists $r_i,l_i\in\NN\cup\{0\}$ such that
\[
C_{i} = F^{l_i}, \ C_{i+1}=F^{r_{i+1}} \quad \textrm{or} \quad C_{i}=F^{l_i}J, \ C_{i+1}= JF^{r_{i+1}}.
\]
This gives potentially several descriptions for $C_i, i\geq 2$, namely $F^{r_i}$ or $JF^{r_i}$, and $F^{l_i}$ or $F^{l_i}J$. By the unique decomposition of \cite{PM} the only option is to have $r_i=l_i$ and $C_i = F^{l_i}$.
Note that this also implies that $C_1 = F^{l_1}$ by the form $C_2$ takes.
Therefore, $C_i = F^{l_i}, i\in\NN$.
\end{proof}

The equivalence relation given by telescope equivalent ordered Bratteli diagrams is then an equivalence relation on symbolic substitutions of powers of the Fibonacci substitution matrix. What is this equivalence relation? The discussion about maximal and minimal paths showed that the two symbolic substitutions on $F$ itself are not equivalent. Therefore, this equivalence is not local indistinguishability.

On the other hand, the equivalence relation given by telescope equivalent ordered Bratteli diagrams is not just equality of substitution. Consider the following example.

\begin{example}
Consider the two symbolic substitutions on two letters corresponding to the Fibonacci matrix $F$:
\[
\begin{array}{l} \sigma_1(a) = ab \\ \sigma_1(b) = a \end{array}
\quad \quad \textrm{and} \quad \quad
\begin{array}{l} \sigma_2(a) = ba \\ \sigma_2(b) = a \end{array}.
\]
Then $\sigma_1^2\circ\sigma_2 = \sigma_2^2\circ\sigma_1$ as both give
\[
a \mapsto ababa \quad \textrm{and} \quad b \mapsto aba.
\]
However, $\sigma_2\circ\sigma_1^2 \neq \sigma_1\circ\sigma_2^2$ as
\begin{align*}
\sigma_2\circ\sigma_1^2(a) = baaba & \neq abaab = \sigma_1\circ\sigma_2^2(a)
\\ \sigma_2\circ\sigma_1^2(b) = baa & \neq aab = \sigma_1\circ\sigma_2^2(b)
\end{align*}
This also says that $(\sigma_2\circ\sigma_1^2)^k \neq (\sigma_1\circ\sigma_2^2)^k$ for any $k\in \NN$ since the former always maps onto words starting with $b$ and the latter always maps onto words starting with $a$.

Now consider that
\[
(\sigma_2\circ\sigma_1^2)^2 \ = \ \sigma_2\circ(\sigma_1^2\circ\sigma_2)\circ\sigma_1^2
\ = \ \sigma_2\circ(\sigma_2^2\circ\sigma_1)\circ\sigma_1 \ = \ \sigma_2^3\circ\sigma_1^3
\]
and 
\[
(\sigma_1\circ\sigma_2^2)^2 \ = \ \sigma_1\circ(\sigma_2^2\circ\sigma_1)\circ\sigma_2^2
\ = \ \sigma_1\circ(\sigma_1^2\circ\sigma_2)\circ\sigma_2 \ = \ \sigma_1^3\circ\sigma_2^3
\]
Hence, we can define $\tau_{2i-1} = \sigma_1^3$ and $\tau_{2i} = \sigma_2^3$ for $i\in\NN$ which gives that
\[
\tau_{2i}\circ\tau_{2i-1} = (\sigma_2\circ\sigma_1^2)^2 \ \ \textrm{and} \ \ \tau_{2i+1}\circ\tau_{2i} = (\sigma_1\circ\sigma_2^2)^2
\]
with $\tau_1$ having the same substitution matrix as $\sigma_2\circ\sigma_1^2$. Therefore, this implies that $\sigma_2\circ\sigma_1^2$ and $\sigma_1\circ\sigma_2^2$ have telescope equivalent ordered Bratteli diagrams.
\end{example}

\section*{Acknowledgements} The first author was supported by the NSERC USRA 2021-562929 and both authors were supported by NSERC Discovery Grant 2019-05430. The second author would like to thank Nicolae Strungaru for several helpful discussions.

\section*{Competing interests:} The authors declare none.

\end{document}